\documentclass[a4paper,12p]{article}
\usepackage[bookmarks=false]{hyperref}
    \hypersetup{colorlinks,
      linkcolor=blue,
      citecolor=blue,
      urlcolor=blue}
\usepackage{amsmath}
\usepackage{mathtools}
\DeclarePairedDelimiter\ceil{\lceil}{\rceil}
\DeclarePairedDelimiter\floor{\lfloor}{\rfloor}

\usepackage[figuresright]{rotating}
\usepackage{amsmath}
\usepackage{amsthm}
\usepackage{amssymb}
\usepackage{amsfonts}
\usepackage{fullpage}
\usepackage{cleveref}
\usepackage{microtype}
\usepackage{url}
\usepackage{tabu}
\usepackage{diagbox}
\usepackage{multirow}
\usepackage{makecell}
\usepackage{todonotes}
\usepackage{subcaption}
\usepackage{verbatim}
\usepackage{wasysym}
\usepackage{caption}
\usepackage{tabularx}
\usepackage{array}
\usepackage{enumitem}
\usepackage{setspace}
\usepackage{mathtools}
\usepackage{microtype}

\newcommand{\ex}{{\rm ex}} 
\newcommand{\bo}{{\rm bot}} 
\newcommand{\Bo}{{\rm Bot}}
\newcommand{\p}{{\rm p}}
\newcommand{\bbal}{{\rm bbal}} 
\newcommand{\Bbal}{{\rm Bbal}} 

\newtheorem{theorem}{Theorem}

\newtheorem{definition}[theorem]{Definition}

\newtheorem{prop}[theorem]{Proposition}
\newtheorem{corollary}[theorem]{Corollary}
\newtheorem{claim}{Claim}

\newcommand{\claimproof}[1]{\noindent\emph{Proof of Claim~\ref{#1}.} }
\newcommand{\claimqed}{\hfill{$\rhd$}}

\title{Unavoidable patterns in $2$-colorings of the complete bipartite graph}

\author{
	Adriana Hansberg $^\ddagger$ 
        Denae Ventura \thanks{Corresponding author} $^{,\S}$ 
	\\ \\ \\
	$^\ddagger$ Instituto de Matem\'aticas, UNAM Juriquilla, Querétaro 76230, Mexico.
	\\
	$^\S$ University of California at Davis, One Shields Avenue, Davis, CA 95616, USA}

\providecommand{\keywords}[1]
{
  \small	
  \textbf{\textit{Keywords---}} #1
}

\begin{document}

\maketitle
\begin{abstract}
We determine the colored patterns that appear in any $2$-edge coloring of $K_{n,n}$, with $n$ large enough and with sufficient edges in each color. 
We prove the existence of a positive integer $z_2$ such that any $2$-edge coloring of $K_{n,n}$ with at least $z_2$ edges in each color contains at least one of these patterns. We give a general upper bound for $z_2$ and prove its tightness for some cases. 
We define the concepts of bipartite $r$-tonality and bipartite omnitonality using the complete bipartite graph as a base graph. 
We provide a characterization for bipartite $r$-tonal graphs and prove that every tree is bipartite omnitonal. Finally, we define the bipartite balancing number and provide the exact bipartite balancing number for paths and stars.
\end{abstract}

\keywords{Ramsey, Zarankiewicz, unavoidable patterns, balanceable graph, omnitonal graph}








\section{Introduction}


Ramsey's theorem \cite{R30} states that, for any integer $t \ge 3$, there exists an integer $R(t)$ such that if $n \ge R(t)$, then any $2$-coloring of the edges of a complete graph of order $n$ contains a monochromatic copy of a complete graph on $t$ vertices. On the other side, Turán's theorem says that, every $n$-order graph having more than $\left( 1- \frac{1}{t}\right) \frac{n^2}{2}$ edges contains a copy of a complete graph on $t$ vertices. There are bipartite versions of these theorems, too. Indeed, for any integer $t \ge 3$, there is the bipartite Ramsey number $BR(t)$, which is the minimum number $n$ for which any $2$-coloring of the complete bipartite graph $K_{n,n}$ with $n$ vertices on each partition set produces a monochromatic copy of a $K_{t,t}$ \cite{BeSch}. Also, the bipartite version of Turán's theorem, known as the Zarakievicz problem, is given by the K\H{o}vári-Sós-Turán theorem that shows that any bipartite graph with $n$ vertices on each side and at least $(t-1)^{\frac{1}{t}}n^{2-\frac{1}{t}} + \frac{1}{2}(t-1)n$ edges contains a copy of $K_{t,t}$ \cite{KST54}. There are multiple generalizations of these theorems, with graphs different from the complete ones, asymmetric versions, with more colors, in hypergraphs, etc. See~\cite{soifer_book} for a general panorama on Ramsey Theory as well as~\cite{bol} for questions concerning extremal settings as in Turán's theorem. More recent results in these areas can be found in \cite{char, conlon, foxden, FoSu08, gra, ramtur}.

In a sense, all the theorems mentioned above talk about similar things: patterns that are unavoidable under certain circumstances. In Ramsey's approach, we have different colors and look for monochromatic substructures. In Turán's approach, one aims to find the minimum density conditions that let certain patterns emerge. If we are interested in the unavoidability of patterns having multiple colors, a natural way of approaching such questions is to consider colorings of the edges of a large enough base graph (like the complete graph or the complete bipartite graph) where a large enough density in each color is guaranteed. Indeed, without asking for a minimum amount of edges in each color, the only patterns that could be unavoidable would be monochromatic. Having this in mind, Bollobás (see \cite{CM08}) conjectured that, for any $\varepsilon > 0$, there is an integer $N(\varepsilon, t)$ such that any $2$-coloring of a complete graph $K_n$ with at least $\varepsilon {n \choose 2}$ edges in each color contains a copy of one of the following $2$-edge colored complete graphs $K_{2t}$: either one where one of the colors induces two complete graphs of order $t$
or one where one color induces a complete graph of order $t$. This conjecture was confirmed in~\cite{CM08} and an asymptotic sharp bound of $N(\varepsilon,t) \le \varepsilon^{-ct}$ for some absolute constant $c$ was later established in \cite{FoSu08}. In \cite{CHM19}, the Turán-type question of this setting was explored and it was shown that, for $n$ large enough, a subquadratic amount $n^{2-1/m(t)}$ (where $m(t)$ is a monotonically increasing function on $t$) of edges in each color is sufficient to guarantee the existence of the above patterns. Later, it was shown in \cite{GiNa19} that $m(t) = t$ works, and, conditioned to the well known conjecture of K\H{o}vári, Sós, and Turán that gives a lower bound for the extremal number $\ex(n, K_{t,t})$, it was shown that $n^{2-1/t}$ is the correct order of magnitude. Multiple color versions were also considered in \cite{BHMM, BLM, GiHa}, and a $2$-color asymmetric version as well as an exhaustive study of all unavoidable patterns depending on the order of magnitude assumed for the minimum number of edges in each color can be found in \cite{CHM_evol}. 

We note at this point that the patterns described above are the only possible colored complete graphs that one can expect (where we are thinking of the unbalanced versions, too). This is because, for an infinite number of $n$'s, there is a way of coloring a complete graph $K_n$ with half of its edges in each of the colors, and such that the edges of one of the colors induce a complete graph or such that the set of vertices can be partitioned into two monochromatic cliques, all whose shared edges are from the other color \cite{CHM_K4}. It is evident that, in these two colorings of $K_n$, the only two-colored patterns that can be found are basically those defined by Bollobás (up to proportions). If we consider now some arbitrary graph $G$ provided with a $2$-coloring of its edges, then it will not necessarily be contained in every $2$-coloring of the complete graph. For instance, it has to \emph{fit} inside the two colorings mentioned above. However, if it fits in both, then, because precisely these patterns are unavoidable, then the so-colored $G$ will appear in every $2$-edge coloring of the complete graph $K_n$ provided $n$ is sufficiently large and the coloring has sufficiently many edges from each color. Instead of considering one graph with one certain coloring, the authors from \cite{CHM19} focused on graphs that can be unavoidably found with half its edges in one color and half its edges in the other color, once we have a sufficiently large complete graph with sufficiently many edges in each of the colors. The graphs having this property are called \emph{balanceable}. They also defined \emph{omnitonal graphs}, which are graphs that can be found, in every tonal variation (i.e. with exactly $r$ red edges for every $0 \le r \le m$, where $m$ is the edge number of the graph in question) in any $2$-coloring of the complete graph with sufficiently many edges in each color. More results about balanceable graphs can be found in \cite{CHM19, BHMM, CHM_K4, CLZ19, DEHV, DHV20}.


In this paper, we begin the study of unavoidable patterns in $2$-edge colorings of complete bipartite graphs. 

A $2$-coloring of the edges of $H$ is a function $f:E(H)\to \{r,b\}$ which associates each edge to one of two colors, $r$ (red) or $b$ (blue). Notice that we can associate every $2$-coloring of the edges of $H$ to a partition $E(H )= R\sqcup B$, where $R = f^{-1}(r)$  and $B = f^{-1}(b)$. We will call the edges in $R$ \textit{red} and the edges in $B$ \textit{blue}. Throughout this text, we will standardly refer to a $2$-coloring of the edges of $H$ simply as a \textit{$2$-coloring} of $H$.
Let $G=(V,E)$ be a graph. We say that a $2$-coloring of $H$ contains a \textit{balanced copy} of $G$ if we can find a copy of $G$ in $H$ such that $E$ can be partitioned into two sets $(E_1 , E_2)$ with $E_1 \subseteq R$, $E_2 \subseteq B$ and such that $\left| |E_1|-|E_2| \right|\leq 1$. If $|E|$ is even, then the balanced copy of $G$ has exactly half of its edges in each chromatic class.

\section{Unavoidable patterns}\label{sec-unav-pat-knn}

Given a graph $H$ and a positive integer $n$, the \emph{extremal number} ${\rm ex}(n,H)$ is defined as the maximum number of edges a graph of order $n$ can have without having a copy of $H$ as a subgraph. In the bipartite context, the \emph{Zarankiewicz number} $z(n,H)$ is the maximum number of edges that a bipartite graph $G = (V \cup W, E)$ can have such that $|V| =|W| = n$ and $G$ has no copy of $H$ (where $H$ has to be assumed bipartite here). We set $z(n,t)$ for $z(n,K_{t,t})$. As mentioned in the Introduction, the K\H{o}vári-Sós-Turán theorem \cite{KST54} yields

\begin{equation}\label{eq_Zarank_nr}
   z(n,t) < (t-1)^{\frac{1}{t}}n^{2-\frac{1}{t}} + \frac{1}{2}(t-1)n.
\end{equation} 

Also note that the following result is deduced by a Dependent random choice lemma given in \cite{FoSu_DRCh}.

\begin{prop}[Dependent random choice]\label{prop-dep}
For all $m$ and $t$ positive integers, there exists a $C=C(m,t)>0$ such that any graph $G$ on $2n$ vertices with $e(G)\geq C (2n)^{2-\frac{1}{t}}$ edges contains a set $S\subset V(G)$ of $m$ vertices where every subset $X\subset S$ such that $|X|=t$ has at least $m$ common neighbors.
\end{prop}

We now state the main result of this section where we use $K_{n,n}$ as a base graph to look for unavoidable patterns in $2$-edge colorings of its edges. 

\begin{theorem}\label{thm-unav-patt}
Let $t$ be a positive integer. For all large enough $n$ and $T\geq t$, there exists a positive integer  $z_2 =z_2 (n,t)= \mathcal{O} (n^{2-\frac{1}{t}})$ such that any coloring $E(K_{n,n})= R\sqcup B$ with at least $z_2 (n,t)$ edges in each color class contains a colored copy of $K_{T,t+T}$
where one color forms a graph isomorphic to $K_{T,T}$, or a copy of $K_{t,2T}$  where one color forms a graph isomorphic to $K_{t,T}$.
\end{theorem}

\begin{proof}

Let $K_{n,n}$ be the complete bipartite graph with partition sets $A$ and $B$, $n$ large enough and let $m$ and $T$ be integers such that $n> m\geq T\geq t$ and 
\begin{equation}\label{bound}
\frac{m^{2}}{2}\geq (T-1)^{\frac{1}{T}}n^{2-\frac{1}{T}} + \frac{1}{2}(T-1)n\geq  (t-1)^{\frac{1}{t}}n^{2-\frac{1}{t}} + \frac{1}{2}(t-1)n
\end{equation}

holds. Let $z_2=C(2n)^{2-\frac{1}{t}}$  where $C=C(m,t)$ is large enough. Let $E(K_{n,n}) = R\sqcup B$ be any $2$-edge coloring with at least $z_2$ edges in each color class. This is possible because $z_2(n,t)=o(n^{2})$. By Proposition~\ref{prop-dep}, the amount of blue edges guarantees that there is a set of vertices $S_1$ in $A$ (without loss of generality) with $|S_1|=m$ and such that any subset $T_1\subseteq S_1$ of order $t$ must have a common blue neighborhood $M_1$ of at least $m$ neighbors. 

A similar argument can be made for the red edges. There is also a set of vertices $S_2$ with $|S_2|=m$ such that any subset $T_2\subseteq S_2$ of order $t$ must have a common red neighborhood $M_2$ of at least $m$ neighbors. We will divide the proof in two cases, one being when $S_2$ is contained in $A$ and the other when $S_2$ is contained in $B$. Notice that $S_2$ cannot intersect both $A$ and $B$, because that would yield the existence of edges in one of the parts $A$ or $B$.


If $S_2 \subset A$, take any subset $T_1\subset S_1$ of order $t$. The set $T_1$ has a common blue neighborhood $M_1$ of order $m$ in $B$. Now we will use the edges induced by $S_2$ and $M_1$ to find a monochromatic $K_{T,T}$. By the pigeonhole principle, there are at least $\frac{m^2}{2}$ edges induced by $S_2$ and $M_1$ in one color class. Because of \cref{bound}, the amount of these edges is an upper bound of $z(m,T)$ (by \cref{eq_Zarank_nr}). Hence there is a monochromatic $K_{T,T}$ in the graph induced by $S_2 \cup M_1$. If this $K_{T,T}$ is red, then the graph induced by $T_1 \cup M_1 \cup S_2$ is a $K_{T,t+T}$, where one color forms a graph isomorphic to $K_{T,T}$. If the $K_{T,T}$ in the graph induced by $S_2 \cup M_1$ is blue and has partition sets $X$ and $Y$ where $X\subset S_2$ and $Y\subset M_1$ then we take a subset $X_1\subseteq X$ with $t$ vertices and by Proposition~\ref{prop-dep} $X_1$ has a red neighborhood $M_2$ of order $m$ in $B$. Notice that the graph induced by $X_1\cup Y \cup M_2$ is a $K_{t,2T}$ where one color forms a graph isomorphic to $K_{T,T}$.

If $S_2 \subset B$, there are at least $\frac{m^2}{2}$ edges of the same color, say red, between the sets $S_1$ and $S_2$ by the pigeonhole principle. Because $\frac{m^{2}}{2}\geq (T-1)^{\frac{1}{T}}n^{2-\frac{1}{T}} + \frac{1}{2}(T-1)n > z(m,T)$, there is a red $K_{T,T}$ contained in the graph induced by the sets $S_1$ and $S_2$. Let $T_1$ and $T_2$ be the partition sets of the red $K_{T,T}$, with $T_1 \subset S_1$ and $T_2 \subset S_2$. By Proposition~\ref{prop-dep} we know that any $t$-subset of $T_1$ has a blue common neighborhood $M_1$ of order $m$. The graph induced by $T_1 \cup T_2 \cup M_1 $ is a $K_{t,2T}$ where one color forms a graph isomorphic to $K_{t,T}$. If the $K_{T,T}$ contained in the graph induced by the sets $T_1$ and $T_2$ is blue, we make a similar argument simply taking any $t$-subset from $T_2$ and its red common neighborhood in $A$. In this case, we also find a $K_{t,2T}$ where one color forms a graph isomorphic to $K_{t,T}$.

\end{proof}

Notice that the parameter $T$, that can be chosen arbitrarily large, allows the existence of a monochromatic complete bipartite graph with a very large side in every $2$-edge coloring of $E(K_{n,n})$ with sufficient edges in each color. As a direct consequence, if $T=t$, we obtain only one colored pattern. This pattern will appear no matter the value of $T$, which is the reason we use the following corollary throughout the rest of the work.

\begin{corollary}\label{coro-unav-pat-Knn}
Let $t$ be a positive integer. For all large enough $n$, there is a positive integer  $z_2= z_2 (n,t)= \mathcal{O} (n^{2-\frac{1}{t}})$ such that any coloring $E(K_{n,n})= R\sqcup B$ with at least $z_2$ edges in each color class contains a colored copy of $K_{t,2t}$
where one color forms a graph isomorphic to $K_{t,t}$.
\end{corollary}

The following proposition employs a result by Alon, Rónyai and Szabó, which states that for a sufficiently large $n$ and integers $s$ and $t$ that satisfy $1\leq s \leq 3 \leq t$ or $t\geq (s-1)!+1$, the value of $\ex(n, G)$ is tight \cite{ARS99}. We use this result to show some values for which $z_2(n,t)$ is tight. Note that any graph $H$ on $n$ vertices and $\ex(n, G)$ edges that does not contain $G$ as a subgraph can be used to define a bipartite graph $H'$ with $n$ vertices in each part and $z(n,G)$ edges with no copies of $G$. This technique is used in the proof of the inequality $2\ex(n,G)\leq z(n, G)$ when $G$ is a bipartite graph. With this in mind, one can give a $2$-edge coloring of $E(K_{n,n})$ where one color, say red, forms the before mentioned graph $H'$ and as a consequence, there are no red copies of $G$ in this coloring. This can also be extended to a coloring of the edges of a complete bipartite graph, which is the idea of the proof of the lower bound of $z_2 (n,t)$.

\begin{prop}

For sufficiently large $n$ and integers $s$ and $t$ that satisfy $1\leq s \leq 3 \leq t$ or $t\geq (s-1)!+1$, \Cref{thm-unav-patt} is tight.

\end{prop}

\begin{proof}
 In \cite{ARS99}, Alon, Rónyai and Szabó exhibit a graph $H$ on at most $n$ vertices which is $K_{s,t}$-free for integers $s$ and $t$ that satisfy $1\leq s \leq 3 \leq t$ or $t\geq (s-1)!+1$. By the commonly known inequality $2\ex(n,G)\leq z(n, G)$ for $G$ bipartite, we can extend this graph $H$ to a bipartite graph $H'$ with at most $n$ vertices in each part which is also $K_{s,t}$-free.  We now give a $2$-edge coloring $E(K_{n,n})=R\sqcup B$ where one color, say blue, forms a copy of $H'$ and the rest of the edges are red. Because $e(H')=\mathcal{O}(n^{2-\frac{1}{s}})$, the edge coloring has at least $\mathcal{O}(n^{2-\frac{1}{s}})$ edges in each color. Therefore, this coloring does not contain blue copies of $K_{s,t}$ when $s$ and $t$ satisfy $1\leq s \leq 3 \leq t$ or $t\geq (s-1)!+1$. Note that this coloring does not contain colored copies of $K_{T,s+T}$ where one color forms a graph isomorphic to $K_{s,T}$, or a copies of $K_{s,2T}$ where one color forms a graph isomorphic to $K_{s,T}$ for $T\geq t$ and the stated values of $s$ and $t$. We have reached the desired result.

\end{proof}

\section{Bipartite omnitonality}\label{sec-omnitonality-knn}

We now define the concept of bipartite $r$-tonality and bipartite omnitonality on the base graph $K_{n,n}$, along with a characterization for bipartite $r$-tonal graphs and a proof that every tree is bipartite omnitonal.

\begin{definition}
Let $G$ be a bipartite graph and $r$ an integer with $0\leq r \leq \lfloor \frac{e(G)}{2}\rfloor$. Let $\bbal_{r}(n,G)$ be the minimum integer, if it exists, such that every $2$-edge coloring $E(K_{n,n})=R\sqcup B$ with $\min\{|R|, |B|\}> \bbal_{r}(n,G)$ contains either an $(r, e(G)-r)-$colored copy of $G$, or an $(e(G)-r,r)-$colored copy of $G$.  If $\bbal_{r}(n,G)$ exists for every $n$ sufficiently large, we say that $G$ is \emph{bipartite $r$-tonal}.
\end{definition}

Now we can state the definition of bipartite omnitonality.

\begin{definition}
Let $G$ be a bipartite graph and  let $\bo(n,G)$ be the minimum integer, if it exists, such that any $2$-coloring $E(K_{n,n})=R\sqcup B$ with $\min \{|R|, |B|\}> \bo(n,G)$ contains an $(r,b)-$colored copy of $G$ for any $r\geq 0$ and $b\geq 0$ such that $r+b=e(G)$. If $\bo(n,G)$ exists for a large enough $n$, we say that $G$ is \emph{bipartite omnitonal}. For a bipartite omnitonal graph $G$, let $\Bo(n,G)$ be the family of graphs $H$ on exactly $\bo(n,G)$ edges such that there is a coloring $E(K_{n,n})=R\sqcup B$ with $|R|=\bo(n,G)$ and with no $(r,b)$-colored copy of $G$ for some pair $r,b\geq 0$ with $r+b=e(G)$ and such that the graph induced by $R$ is isomorphic to $H$.
\end{definition}

Observe that if $G$ is bipartite $r$-tonal, then $\bbal_r (n,G)\leq \frac{n^2}{2}$. Also $\bo(n,G) \leq \frac{n^2}{2}$ when $G$ is a bipartite omnitonal graph.

\begin{theorem}\label{thm-carac-b-r-tonal}
Let $G$ be a bipartite graph and let $r$ be a positive integer with $r\leq \lfloor \frac{e(G)}{2} \rfloor$. Then $G$ is bipartite $r$-tonal if and only if there is a set of vertices $U\subseteq V(G)$ such that $e(U,N(U) )=r$ and $U$ is contained in one of the partition sets of $G$. 
\end{theorem}

\begin{proof}
Let $G$ be bipartite $r$-tonal and $n$ even and sufficiently large such that $\bbal_r (n, G)$ exists. Let the partition sets of $K_{n,n}$ be $X$ and $Y$ with $Y=Y_r \sqcup Y_b$ and $|Y_r|=\frac{n}{2}$. Let $E(K_{n,n})=R\sqcup B$ be a $2$-edge coloring where the red edges are all possible edges between the sets $X$ and $Y_r$ and the blue edges are all possible edges between $X$ and $Y_b$.
Because $G$ is bipartite $r$-tonal and $|R|=|B|=\frac{n^2}{2}$, we can guarantee the existence of a colored copy of $G$ where one color class, say red, contains exactly $r$ edges. Let $U=V(G)\cap Y_r$. This must be a set of independent vertices that satisfies that $e_G(U,N_G(U))=r$ and $U$ is clearly contained in one of the partition sets of $K_{n,n}$, and hence, of $G$.

Conversely, let $G$ be a bipartite graph that contains a set of vertices $U$ such that $e_G(U,N_G(U) )=r$ and $U$ is contained in one of the partition sets of $G$. Let $t=n(G)$ and let $E(K_{n,n})=R\sqcup B$ be a $2$-edge coloring, where $n$ is large enough, with more than $z_2 (n,t)$ edges in each color, just as in Theorem~\ref{thm-unav-patt}. By \Cref{coro-unav-pat-Knn}, there is a colored copy of a $K_{t,2t}$ where one color forms a $K_{t,t}$. Let $X=X_b \cup X_r$ and $Y$ be the partition sets of this $K_{t,2t}$, where $|X_b|=|X_r|=|Y|=t$ and such that $E(X_b , Y) \subseteq B$ and $E(Y_r , X) \subseteq R$.

Let $X_G$ and $Y_G$ be the partition sets of $V(G)$ such that $U\subseteq X_G$. Now we can embed $G$ in this $K_{t,2t}$ by placing the set $U$ inside $X_b$, the set $X_G\setminus U$ inside $X_r$ and the set $Y_G$ inside $Y$. This copy of $G$ has exactly $r$ blue edges. Hence, $G$ is bipartite $r$-tonal.
%
%
%

\end{proof}

Similarly as in \Cref{thm-carac-b-r-tonal}, we can prove a characterization of bipartite omnitonal graphs.

\begin{theorem}\label{thm-carac-b-omnitonal}
A bipartite graph $G$ is bipartite omnitonal if and only if for every $0\leq r\leq e(G)$ there is a set of vertices $U_r \subseteq V(G)$ such that $e_G(U_r ,N_G(U_r) )=r$ and $U_r$ is contained in one of the partition sets of $G$. 
\end{theorem}

\begin{proof}
Suppose $G$ is bipartite omnitonal. Let $n$ be even and large enough so that $\bo(n,G)$ exists. Consider an edge coloring $E(K_{n,n})=R\sqcup B$ where one color forms a $K_{n,\frac{n}{2}}$. Since $G$ is bipartite omnitonal and $\bo(n, G)\leq \frac{n^2}{2}=|R|=|B|$, there must be a copy of $G$ in $K_{n,n}$ with $r$ red edges for every $0\leq r \leq e(G)$. This implies that, for every $0\leq r\leq e(G)$, there is a set of vertices $U_r$ such that $e_G(U_r ,N_G(U_r) )=r$ and $U_r$ is contained in one of the partition sets of $G$.

Conversely, suppose that for every $0\leq r\leq e(G)$, there is a set of vertices $U_r \subseteq V(G)$ such that $e_G(U_r ,N_G(U_r) )=r$ and $U_r$ is contained in one of the partition sets of $G$. Let $E(K_{n,n})=R\sqcup B$ be a $2$-edge coloring with more than $z= z_2(n, t)$, where $t=n(G)$ and $z$ is like in \Cref{coro-unav-pat-Knn}. Therefore, if $n$ is sufficiently large, there is a copy of $K_{t,2t}$ (with partition sets $X$ and $Y$) where one color, say red, forms a graph isomorphic to $K_{t,t}$ (with partition sets $X$ and $Y_1 \subseteq Y$). Because there is a set of vertices $U_r \subseteq V(G)$ such that $e_G(U_r ,N_G(U_r) )=r$ and $U_r$ is contained in one of the partition sets of $G$, we can place $U_r$ in $Y_1$ to find a copy of $G$ with $r$ red edges and $e(G)-r$ blue edges for every $0\leq r\leq e(G)$. This concludes the proof.
\end{proof}

As in the setting where $K_n$ is the base graph \cite{CHM19} it turns out that all trees are bipartite omnitonal.

\begin{theorem}\label{trees-b-omnit}
Every tree $T$ is bipartite omnitonal.
\end{theorem}

 \begin{proof}
 Let $T$ be a tree with $e(T)=k$. In view of \Cref{thm-carac-b-r-tonal}, we prove that, for every $0 \leq r \leq e(T)$, there is a set $U_r$ in one of the partition sets of $T$ such that $e_T (U_r, N_T (U_r))=r$ by induction on $k$. 
 
If $k=1$, we may that $U_0$ as the empty set and $U_1$ as a single vertex. Both sets satisfy that $e_T (U_r, N_T (U_r))=r$ for $r=0,1$. Let $v$ be a leaf in $V(T)$ and let $u$ be its only neighbor. By the induction hypothesis, the tree $T'=T-v$ is bipartite omnitonal, which means that, for every $0 \leq r\leq k-1$, there is a set $U'_r$ contained in one of the partition sets of $T'$ such that $e_{T'}(U'_r , N_{T'}(U'_r))=r$.
 
For each $r$, there are two cases that can happen. If $u\in U'_r$, then $v\in V(T)\setminus U' _r$, and we may take $U_{r+1} = U' _r$ so that $e_T (U_{r+1} , N_T (U_{r+1} ))=r+1$. The second case is that $u\in V(T)\setminus U' _r$. In this case, we let $U_{r+1}= U' _r \cup \{v\}$. Hence for every $1\leq r \leq k$, there is a set $U_r$ that satisfies the theorem. The only case missing is when $r=0$. In this case we can take $U_0$ as the empty set because $e_T (U_0 , N_T (U_0))=0$. This concludes the proof.

 \end{proof}

\section{Bipartite balanceability and the bipartite balancing number}\label{sec-bip-balanceability-bip-bal-num}
The bipartite balancing number is the analogue of the balancing number introduced in \cite{CHM19} but defined on the base graph $K_{n,n}$. In this section, we provide the bipartite balancing number for paths and stars.
\begin{definition}
If there exists an integer $k=k(n)$ such that, for $n$ large enough, every $2$-edge coloring $R\sqcup B$ of $E(K_{n,n})$ with more than $k$ edges in each color class contains a balanced copy of $G$, then we say $G$ is \emph{bipartite balanceable}. The smallest such $k$ is called the bipartite balancing number of $G$ and it is denoted as \emph{$\bbal(n,G)$}. For a bipartite balanceable graph $G$, let $\Bbal(n,G)$ be the family of graphs with exactly $\bbal (n, G)$ edges such that a $2$-edge coloring $E(K_{n,n} )= R\sqcup B$ with exactly $\bbal (n, G)$ edges in one color contains no balanced copy of $G$ if and only if the graph induced by the red edges or by the blue edges is isomorphic to some $H\in \Bbal (n,G)$.
\end{definition}

Clearly, $\bbal(n,G) \le \bo(n,G)$ for any bipartite omnitonal graph $G$. Hence, the following corollary follows directly from \Cref{trees-b-omnit}.

\begin{corollary}
Every tree $T$ is bipartite balanceable.
\end{corollary}

In the following, we determine the value of $\bbal(n, G)$ for paths and stars. 

\begin{theorem}\label{thm-bbal-Pk}
Let $n \ge 2$ and  $k\geq2$ be integers, where $k$ is even. If $n> \frac{k-2}{2}\left(\frac{k}{2}-\left\lfloor\frac{k-2}{4}\right\rfloor\right)$, then
\[\bbal(n,P_{k} )=\bbal(n,P_{k+1} ) = \left\lfloor\frac{k-2}{4}\right\rfloor n + 0^{\alpha},\]
where $\alpha \in \{0,2\}$ is such that $k\equiv \alpha \pmod{4}$. Moreover, if $s=\left\lfloor\frac{k-2}{4}\right\rfloor$, then $\Bbal(n, P_k )= \Bbal(n, P_{k+1} )=\{K_{s,n}\}$, if $k\equiv 2\pmod{4}$, and  $\Bbal(n, P_k )=\Bbal(n, P_{k+1} )= \{H_{s,n}\}$, if $k\equiv 0\pmod{4}$, where $H_{s,n}$ is isomorphic to the graph $K_{s,n}$ with one extra pendant vertex adjacent to a vertex of the partite set with $n$ vertices.
\end{theorem}

\begin{proof}
Let $n \ge 2$ and  $k\geq2$ be integers, where $k$ is even, and $n> \frac{k-2}{2}\left(\frac{k}{2}-\floor*{\frac{k-2}{4}}\right)$. Let $p(n,k)=\floor*{\frac{k-2}{4}}n+0^{\alpha}$ where $\alpha \in \{0,2\}$ is such that $k\equiv \alpha \pmod{4}$. We determine the bipartite-balancing number of paths with an even number of edges. Once this is determined, the odd case is straightforward, and it is discussed at the end of the proof. We begin by proving that a $2$-edge coloring $E(K_{n,n})=R\sqcup B$ with more than $p(n,k)$ edges in each color is satisfiable. We proceed to prove that

$$e(K_{n,n})=n^2 > 2p(n,k)$$

for all $n> \frac{k-2}{2}\left(\frac{k}{2}-\floor*{\frac{k-2}{4}}\right)$. If $k\equiv 0\pmod{4}$, then $2p(n,k)=2(\floor*{\frac{k-2}{4}}n+0^{\alpha})=2(\frac{k-4}{4}n+1)<n^2$, because $n>\frac{k-4}{2}+2$. If $k\equiv 2\pmod{4}$, then $2p(n,k)=2(\floor*{\frac{k-2}{4}}n+0^{\alpha})=2(\frac{k-2}{4})n<n^2$, because $n>\frac{k-2}{2}$.  

We proceed now to prove, by induction on $k$, that any $2$-coloring of $K_{n,n}$ with at least $\p(n,k)$ edges in each color has either a balanced $P_k$ or the edges in one of the colors induce the graph $B(k)$, where $B(k)= H_{s,n}$, if $\equiv 0\pmod{4}$, or $B(k)=  K_{s,n}$, if $k\equiv 2\pmod{4}$. We assume there is no balanced $P_k$ in order to arrive to one of these extremal colorings, each of which satisfies that every copy of $P_k$ has at most $\frac{k}{2}-1$ edges in one color class.

When $k=2$, any $2$-edge coloring of $K_{n,n}$ with at least one edge in each color is enough to find a balanced $P_2$, hence $\bbal(n,P_2)=0$, while $K_{0,n}$ represents the extremal coloring which in this case is the empty graph. Let $k > 2$ be even, and assume the statement of the theorem is true for each even integer $k'$ with $2 \le k' < k$.

Suppose that $E(K_{n,n})=R\sqcup B$ is a $2$-edge coloring with at least $\p(n,k)$ edges in each color. It is a simple matter to check that the conditions for $k' = k-2$ are satisfied. If $k\equiv 0\pmod{4}$ then $\p(n,k)=\lfloor \frac{k-2}{4} \rfloor n+1=\frac{k-4}{4}n+1=\lfloor \frac{k-4}{4} \rfloor n+1=\p(n, k')$, and if $k\equiv 2\pmod{4}$ then $\p(n,k)=\lfloor \frac{k-2}{4} \rfloor n=\frac{k-2}{4}n>\lfloor \frac{k-4}{4} \rfloor n=\p(n,k')$. We can assume that, by the induction hypothesis, there is a balanced copy of $P_{k-2}$, say $P=v_1 v_2 \cdots v_{k-1}$. Let $K_{n,n}$ have partite sets $V$ and $W$, where $V_P = V\cap V(P)$ and $W_P = W \cap V(P)$. We will show that either there is a balanced $P_k$ or we will have one of two extremal colorings depending on the value of $k$ modulo $4$.

We begin by analyzing the structure of the coloring outside of $P$. Note first that all of the edges that come out of $V(P)$ with an endpoint in $\{v_1, v_{k-1}\}$ must have the same color. Otherwise, if $v_1v$ is red and $v_{k-1}u$ is blue with $u,v$ being distinct vertices outside of $V(P)$, then $v v_1Pv_{k-1} u$ forms a balanced $P_k$. We assume, without loss of generality, that all such edges are blue, that is 

 \begin{equation}\label{start_end_vertices}
    E(\{v_1, v_{k-1}\}, (V \cup W)\setminus V(P)) \subseteq B.
\end{equation}

By a similar argument, we can conclude that 
\begin{equation}\label{edges_outside}
    E(V\setminus V_P, W\setminus W_P)\subseteq B.
\end{equation}

Otherwise, if $xy\in E(V\setminus V_P , W\setminus W_P)\cap R$, then $xyv_1Pv_{k-1}$ makes a balanced $P_k$.

\begin{claim}\label{claim-0}
Any edge with one endpoint outside of $V(P)$ that is incident to a blue edge in $P$, must be blue. 
\end{claim}

\claimproof{claim-0}
We will show now that every time there is a blue edge in $P$, no red edges can come out of it. If $x\in V_P$, $y\in W_P$, $w\in W\setminus W_P$, $xy\in E(P)\cap B$ and $vy\in R$ for some $v\in V\setminus V_P$, then $xPv_1wvyPv_{k-1}$ makes a balanced $P_k$ (recall that $v_1 w$ and $wv$ are blue). If now we take $x\in W_P$, $y\in V_P$, $xy\in E(P)\cap B$, $wy\in R$ for some $w\in W\setminus W_P$ and some $w'\in W\setminus W_P$ with $w'\neq w$, then $xPv_1wyPv_{k-1}w'$ forms a balanced $P_k$. \claimqed\\

Let $V_P ^r$ be the set of vertices in $V_P$ that make at least one red edge with a vertex in $W\setminus W_P$, and let $W_P ^r$ be the set of vertices in $W_P$ that make at least one red edge with a vertex in $V\setminus V_P$. Therefore, by \Cref{claim-0} all vertices in $V_P^r \cup W_P^r$ are incident to two red edges from $P$.
By definition, 

\begin{equation}\label{eq-00}
e_R (v,W\setminus W_P)\neq 0, \hspace{3pt}
e_R (u,V\setminus V_P)\neq 0 
\end{equation}

for all $v\in V_P^r$ and for all $u\in W_P^r$.

Note that every blue edge in $P$ has one endpoint in $V_P\setminus V_P^r$ and  the other endpoint in $W_P\setminus W_P^r$. Also every vertex in $V_P\setminus V_P^r$ and $W_P\setminus W_P^r$ is either incident to some blue edge in $P$, or it is incident to two red edges in $P$ and makes only blue edges with $(V\setminus V_P)\cup (W\setminus W_P)$ (otherwise, it would be in $V_P^r\cup W_P^r$). Along with \Cref{claim-0}, we can conclude that all vertices in $V_P^r \cup W_P^r$ are incident to two red edges from $P$, and

\begin{equation}\label{eq-2}
E(V_P \setminus V_P ^r , W\setminus W_P) \cup E(W_P \setminus W_P ^r , V\setminus V_P)\subseteq B.
\end{equation}

\begin{claim}\label{claim-1}
All red edges not in $E(P)$ incident to a red edge $e$ in $P$ must be incident to the same endpoint in $e$.  Therefore, every red edge in $E(P)$ has an endpoint with only blue incident edges.
\end{claim}

\claimproof{claim-1}
If $xy\in E(P)\cap R$ and there are $u,v\in (V\cup W)\setminus V(P)$ such that $ux, vy\in R$, then $v_1PxuvyPv_{k-1}$ forms a balanced $P_k$ since we already know that $uv\in B$. \claimqed\\

Notice that $|V_P ^r\cup W_P^r|\leq \lfloor \frac{k-2}{4} \rfloor$ as from each vertex in $V_P ^r\cup W_P^r$ there are two red edges from $P$ and $P$ has $\frac{k-2}{2}$ red edges. On the other hand, if $|V_P^r \cup W_P^r|<\lfloor \frac{k-2}{4} \rfloor$, say $|V_P^r\cup W_P^r|= |V_P^r|+|W_P^r|=\lfloor \frac{k-2}{4} \rfloor -q$ for some $1\leq q\leq \floor*{\frac{k-2}{4}}$, then the number of red edges would not satisfy the hypothesis as we can see in the next bound that employs previous coloring observations.

\begin{align*}
|R|&= e_R (V_P^r, W\setminus W_P) + e_R (W_P^r, V \setminus V_P) + e_R (V_P , W_P)\\
&\leq |V_P^r|\cdot |W\setminus W_P|+|W_P^r|\cdot |V\setminus V_P|+|V_P|\cdot|W_P|\\
&= |V_P^r|\cdot \left(n-|W_P|\right)+\left(\floor*{\frac{k-2}{4}} -|V_P^r| -q\right)\left(n-|V_P|\right)+|V_P|\cdot |W_P|\\
&\leq |V_P^r|\cdot \left(n-|W_P|\right)+\left(\floor*{\frac{k-2}{4}} -|V_P^r| -q\right)\left(n-|W_P|\right)+|V_P|\cdot |W_P|\\
&\leq \left(n-|W_P|\right)\left(|V_P^r|+ \floor*{\frac{k-2}{4}}-|V_P^r| -q\right) +  \frac{k}{2} \cdot \frac{k-2}{2} \\
&< n\left(\floor*{\frac{k-2}{4}}-q\right) + \frac{k^2}{4}\\
&\leq n\left(\floor*{\frac{k-2}{4}}-1\right) +\frac{k^2}{4}
\end{align*}

considering Equations (\ref{edges_outside}) and (\ref{eq-2}). However, this contradicts the fact that $|R|\geq \lfloor \frac{k-2}{4} \rfloor n$ due to the fact that $n> \frac{k-2}{2}\left(\frac{k}{2}-\floor*{\frac{k-2}{4}}\right)$ when $k\geq 4$.

Therefore,
\begin{equation}\label{eq-VPr}
|V_P^r\cup W_P^r|= \floor*{\frac{k-2}{4}}.
\end{equation}

We can conclude that exactly $2\floor*{\frac{k-2}{4}}$ red edges in $P$ are incident to vertices from $V_P^r\cup W_P^r$. Therefore, we arrive to the following claim. 

\begin{claim}\label{claim-11}
Every red edge in $P$ is incident to some vertex in $V_P^r\cup W_P^r$, except for one edge with endpoints in $V_P\setminus V_P^r$ and $W_P\setminus W_P^r$ in the case where $k\equiv 0\pmod{4}$. 
\end{claim}

Now we can analyze the structure of the coloring within $P$ to prove the next inclusion.

\begin{equation}\label{claim-3}
    E(V_P \setminus V_P^r, W_P \setminus W_P^r) \setminus E(P) \subseteq B.
\end{equation}

\textit{Proof of \Cref{claim-3}}
By \Cref{claim-11}, recall that every vertex $v$ in $(V_P\setminus V_P^r)\cup (W_P\setminus W_P^r)$ is either incident to some blue edge in $P$ or it is incident to two red edges in $P$. Additionally, $v$ makes only blue edges with $(V\setminus V_P)\cup (W\setminus W_P)$ by \Cref{eq-2}. Let $v_i,v_j \in (V_P \cup W_P)\setminus (V_P^r \cup W_P^r)$ such that $i<j$. We go over all possible cases assuming $v_iv_j \in R$.  
In each case, we construct a balanced $P_k$.\\

\noindent\textbf{Case 1:} Let $v_iv_{i+1}, v_j v_{j+1}\in B$ with $v_i\in V$ and $v_j\in W$. Let $x,w\in V\setminus V_P$ and $y\in W\setminus W_P$. If $v_1,v_{k-1}\in V_P$, the path $v_1Pv_iv_jPv_{i+1}xyv_{k-1}Pv_{j+1}$ makes a balanced $P_k$. If $v_1,v_{k-1}\in W_P$, the path $wv_1Pv_iv_jPv_{i+1}xv_{k-1}Pv_{j+1}$ makes a balanced $P_k$.

 This case is analogous by symmetry to the case where $v_{i-1}v_{i}, v_{j-1} v_{j}\in B$ with $v_i\in W$ and $v_j\in V$.\\

\noindent\textbf{Case 2:} Let $v_{i-1}v_{i}, v_{j-1} v_{j}\in B$ with $v_i\in V$ and $v_j\in W$. Let $x,y\in W\setminus W_P$ and $w\in V\setminus V_P$. If $v_1,v_{k-1}\in V_P$, the path $v_{i-1}Pv_1xv_{j-1}Pv_iv_jPv_{k-1}y$ makes a balanced $P_k$. If $v_1,v_{k-1}\in W_P$, the path $v_{i-1}Pv_1wyv_{j-1}Pv_iv_jPv_{k-1}$ makes a balanced $P_k$. This case is analogous by symmetry to the case where $v_{i}v_{i+1}, v_{j} v_{j+1}\in B$ with $v_i\in W$ and $v_j\in V$.\\

\noindent\textbf{Case 3:} Let $v_{i}v_{i+1},v_{j-1}v_j \in B$. Let $x,w\in V\setminus V_P$ and $y\in W\setminus W_P$. We can assume that  $v_{i-1}v_i, v_j v_{j+1}\in R$, otherwise, see Cases 1 and 2. Without loss of generality we can assume that, $v_i\in V$ and $v_j \in W$. If $v_1,v_{k-1}\in V_P$, the path $v_1Pv_iv_jPv_{k-1}yxv_{j-1}Pv_{i+1}$ makes a balanced $P_k$. If $v_1,v_{k-1}\in W_P$, the path $wv_1Pv_iv_jPv_{k-1}xv_{j-1}Pv_{i+1}$ makes a balanced $P_k$.\\

\noindent\textbf{Case 4:} Let $v_{i-1}v_{i}, v_{j} v_{j+1}\in B$. We can assume that $v_iv_{i+1},v_{j-1}v_j\in R$, otherwise, see Cases 1 or 2. Without loss of generality, we can assume that $v_i\in V$ and $v_j\in W$. We can suppose $v_{i-1}v_{j+1}\in B$, otherwise, see Case 3. Suppose there is an edge $v_tv_{t+1}\in B$ with $i<t<j$. Let $x,y\in W\setminus W_P$ such that $x\neq y$ and $w\in V\setminus V_P$. 

If $v_t \in V_P$ and $v_1,v_{k-1}\in V_P$, the path $v_{t+1}Pv_jv_iPv_txv_1 P v_{i-1}v_{j+1}Pv_{k-1}y$ makes a balanced $P_k$. If $v_t \in V_P$ and $v_1,v_{k-1}\in W_P$, the path $v_{t+1}Pv_jv_iPv_txwv_1 P v_{i-1}v_{j+1}Pv_{k-1}$ makes a balanced $P_k$.

If $v_t \in W_P$ and $v_1,v_{k-1}\in W_P$, the path $v_1Pv_{i-1}v_{j+1}Pv_{k-1}wxv_{t+1}Pv_jv_iPv_t$ makes a balanced $P_k$. Therefore, all edges in $P$ between $v_i$ and $v_j$ must be red. Because there are no odd cycles in $K_{n,n}$, there are at least two vertices in $V(P)$ between $v_{i}$ and $v_j$. By \Cref{claim-11}, $v_{j-1}\in V_P^r$ or $v_{i+1}\in W_P^r$. Suppose $v_{j-1}\in V_P^r$. There are $x,y\in W\setminus W_P$, with $x\neq y$, and such that $v_{j-1} x \in R$. If $v_1,v_{k-1}\in V_P$, the path $v_{i-1}Pv_1xv_{j-1}Pv_{i}v_j P v_{k-1}y$ makes a balanced $P_k$ (recall $v_1x,v_{k-1}y\in B$). If $v_1,v_{k-1}\in W_P$, the path $v_{i-1}Pv_1wxv_{j-1}Pv_{i}v_j P v_{k-1}$ makes a balanced $P_k$. The case when $v_{i+1}\in W_P^r$ is analogous by symmetry.\\

\noindent\textbf{Case 5:} Let $v_i\in V_P\setminus V_P^r$, $v_j \in W_P\setminus W_P^r$. We consider four subcases. We prove each subcase assuming $v_{i-1}\in W_P\setminus W_P^r$ and $v_{i-1}v_i\in B$ or $v_{i+1}\in W_P\setminus W_P^r$ and $v_iv_{i+1}\in B$, when possible.

\begin{itemize} 
\item \emph{Subcase 5.1: Let $v_{j-1}\in V_P^r$ and $v_{j+1}\in V_P\setminus V_P^r$.} If $v_{i-1}v_i \in B$ for $i\geq 2$, then $v_1Pv_{i-1}xyv_{j-1}Pv_i v_j P v_{k-1}$ makes a balanced $P_k$ for $x\in V\setminus V_P$, and $y\in W\setminus W_P$ such that $v_{j-1}y\in R$.

If $v_i v_{i+1}\in B$ for $i\geq 1$ and $v_1,v_{k-1}\in V_P$, then $xv_1Pv_iv_jPv_{k-1}yv_{j-1}Pv_{i+1}$ makes a balanced $P_k$ for $x,y\in W\setminus W_P$ and $v_{j-1}y\in R$. If $v_1,v_{k-1}\in W_P$, then $v_1Pv_iv_jPv_{k-1}yxv_{j-1}Pv_{i+1}$ is a balanced $P_k$ for $x\in W\setminus W_P$ such that $v_{j-1}x\in R$, and $y\in V\setminus V_P$. 

\item \emph{Subcase 5.2: Let $v_{j-1}\in V_P\setminus V_P^r$ and $v_{j+1}\in V_P^r$.} This implies that $v_{j-2}v_{j-1}\in B$. If $v_{i-1}v_i \in B$ for $i\geq 2$, then $v_1Pv_{i-1}xv_{j-2}Pv_iv_jv_{j-1}yv_{j+1}Pv_{k-1}$ is a balanced $P_k$ for $x\in V\setminus V_P$ and $y\in W\setminus W_P$ such that $v_{j+1}y\in R$.

If $v_iv_{i+1}\in B$ for $i\geq 1$, then $v_1Pv_iv_jv_{j-1}Pv_{i+1}xyv_{j+1}Pv_{k-1}$ is a balanced $P_k$ for $x\in V\setminus V_P$ such that $v_{j-1}x\in R$ and $y\in W\setminus W_P$.

\item \emph{Subcase 5.3: Let $j=k-1$ and $v_{j-1}=v_{k-2}\in V_P^r$.} If $v_{i-1}v_i\in B$ for $i\geq 2$ and $v_1,v_{k-1}\in W_P$, then $xv_1Pv_{i-1}yv_{j-1}Pv_iv_j$ is a balanced $P_k$ for $x\in V\setminus V_P$ and $y\in W\setminus W_P$ such that $v_{j-1}y\in R$. If $v_1,v_{k-1}\in V_P$, then $v_1Pv_{i-1}wyv_{j-1}Pv_iv_j$ is a balanced $P_k$ for $w\in W\setminus W_P$ and $y\in V\setminus V_P$ such that $v_{j-1}y\in R$.

If $v_iv_{i+1}\in B$, for $i\geq 1$ and $v_1,v_{k-1}\in W_P$, then $v_1Pv_iv_j y x v_{j-1}Pv_{i+1}$ is a balanced $P_k$, for $x\in W\setminus W_P$ such that $v_{j-1}x\in R$, and $y\in V\setminus V_P$. If $v_1,v_{k-1}\in V_P$, then $wv_1Pv_iv_j x v_{j-1}Pv_{i+1}$ is a balanced $P_k$ where $w\in V\setminus V_P$.

\item \emph{Subcase 5.4: Let $i=1$, $v_1v_2\in R$, and $v_{j-1}\in V_P^r$.} There are vertices $x\in W\setminus W_P$ and $y\in (V\cup W)\setminus V(P)$ such that $xv_{j-1}\in R$ and $v_{k-1}y\in B$ by \Cref{start_end_vertices}. The path $xv_{j-1}Pv_1v_jPv_{k-1}y$ makes a balanced $P_k$.

There are vertices $x,y\in W\setminus W_P$ such that $xv_{j-2}\in B$ and $v_{j+1}y\in R$ by \Cref{eq-2}. The path $v_{j-1}v_jv_1Pv_{j-2}xyv_{j+1}Pv_{k-1}$ makes a balanced $P_k$. 
\end{itemize}

\claimqed\\

By the previous observations, note that the edges in $E(V\setminus V_P^r,W\setminus W_P^r)$ are all blue if $k\equiv 2\pmod{4}$. If $k\equiv 0\pmod{4}$, the same holds except for one red edge which is in $E(V_P \setminus V_P^r,W_P \setminus W_P^r)$. Also note that, in both cases $E(V_P^r,W_P^r)\subseteq B$. We proceed to prove that one of the sets, $V_P^r$ or $W_P^r$, must be empty. Recall that, 
$|V_P^r|+|W_P^r|=\lfloor \frac{k-2}{4}\rfloor$. 
Without loss of generality suppose that $|V_P^r|\geq |W_P^r|$. If $|W_P^r|\geq 1$, the number of red edges is bounded as follows.
\begin{align*}
    |R|&=e_R(V_P^r,W)+e_R(W_P^r,V)\\ 
    &\leq |V_P^r|(n-|W_P^r|)+|W_P^r|(n-|V_P^r|)\\ 
    &\leq |V_P^r|(n-|W_P^r|)+|W_P^r|(n-|W_P^r|)\\
    &=(|V_P^r|+|W_P^r|)(n-|W_P^r|)\\ 
    &\leq \left\lfloor \frac{k-2}{4} \right\rfloor (n-1)
\end{align*}

This contradicts the fact that $|R|\geq \lfloor \frac{k-2}{4}\rfloor n$. Therefore,  $W_P^r$ must be empty.
Thus, we can conclude that 

\begin{enumerate}
\item[i)] If $k\equiv 0 \pmod{4}$ then the red edges form a $K_{s,n}$ with $s=\lfloor \frac{k-2}{4} \rfloor$ and an extra pendant vertex adjacent to a vertex of the partite set with $n$ vertices. 

\item[ii)] If $k\equiv 2 \pmod{4}$ then the red edges form a $K_{s,n}$, with $s=\lfloor \frac{k-2}{4} \rfloor$.
\end{enumerate}

Observe that $\bbal(n,P_{k+1} )=\bbal(n,P_k )$ for $k\geq 2$ even. This holds in any $2$-edge coloring of $E(K_{n,n})$ where there is a balanced copy of $P_k$, because every balanced $P_k$ can be extended to a balanced $P_{k+1}$ by considering one more edge adjacent to one of its end vertices. Regardless of its color, the extended path $P_{k+1}$ has $\floor*{\frac{k}{2}}$ edges in one color and $\ceil*{\frac{k}{2}}$ in the other. The fact that $\Bbal(n,P_k )=\Bbal(n,P_{k+1} )$ holds, because the extremal coloring associated to $\Bbal(n,P_k )$ also avoids copies of balanced $P_{k+1}$. Hence, the bipartite-balancing number of paths of odd length is also determined. This concludes the proof.

\end{proof}

We now state the exact bipartite balancing number for stars. 

\begin{theorem} Let $n$ and $k$ be integers with $k$ even, $k\geq 2$ and $n\geq \frac{k^2}{2}+k-2$, then
\[\bbal(n,K_{1,k})=\bbal(n,K_{1,k+1})=(k-2)\left(n-\frac{k-2}{4}\right),\]
and $\Bbal(n, K_{1,k})=\Bbal(n, K_{1,k+1})$ contains only one graph which is a bipartite graph $H$ with partition sets $A$ and $B$ where $A=A_1 \sqcup A_2$ and $B=B_1 \sqcup B_2$ with $|A_1|=|B_2|=n-\frac{k-2}{2}$ and $|A_2|=|B_1|=\frac{k-2}{2}$. The edges of $H$ are all possible edges between $A_1$ and $B_1$ and between $A_2$ and $B$.
\end{theorem}

\begin{proof}
Let $n$ and $k$ be integers with $k$ even, $k\geq 2$ and $n\geq \frac{k^2}{2}+k-2$. We determine the bipartite-balancing number for stars of even number of edges. We can easily extend this result for stars of odd number of edges. It is simple to see that $\bbal(n,K_{1,k})=\bbal(n,K_{1,k+1})$ when $k\geq 2$ is even. Note that, in every $2$-edge coloring of $K_{n,n}$ where there is a balanced $K_{1,k}$, we can extend this copy to a balanced $K_{1,k+1}$ by considering one more edge adjacent to the vertex of maximum degree. Regardless of its color, the extended star $K_{1,k+1}$ has $\floor*{\frac{k}{2}}$ edges in one color and $\ceil*{\frac{k}{2}}$ in the other. The fact that $\Bbal(n,K_{1,k})=\Bbal(n,K_{1,k+1})$ holds, because the extremal coloring associated to $\Bbal(n,K_{1,k})$ also avoids copies of balanced $K_{1,k+1}$.

Therefore, we prove the result for stars of even number of edges. Let $a(n,k):= (k-2)(n-\frac{k-2}{4}).$ We begin by proving that a $2$-edge coloring $E(K_{n,n})=R\sqcup B$ with more than $a(n,k)$ edges in each color is satisfiable. We proceed to prove that 
$$e(K_{n,n})=n^2 > 2a(n,k)$$ for all $n\geq  \frac{k^2}{2}+k-2$. 
Note that $n^2 > 2a(n,k)=2(k-2)(n-\frac{k-2}{4})$ is equivalent to $(n-k+2)^2-\frac{(k-2)^2} {2} > 0$ and this is true when $n> (k-2)(1+ \frac{1}{\sqrt{2}} )$ and $k\geq 4$, which holds by the hypothesis.\\

Let $H$ be the bipartite graph described in the theorem statement. Observe that $H$ has exactly $|A_1||B_1|+|A_2||B|= (n-\frac{k-2}{2})\frac{k-2}{2} + \frac{k-2}{2}n=(k-2)(n-\frac{k-2}{4})=a(n,k)$ edges.

Note that any $2$-edge coloring $E(K_{n,n})=R\sqcup B$, where the graph induced by one of the color classes is isomorphic to $H$, does not contain a balanced copy of $K_{1,k}$. If a vertex $v$ in this coloring has at least $\frac{k}{2}$ neighbors in one color, it will have less than $\frac{k}{2}$ neighbors in the other color making the existence of a balanced $K_{1,k}$ impossible. Hence, $\bbal(n,K_{1,k}) \geq a(n,k)$ and $H\in \Bbal(n,K_{1,k})$ is proved.

To prove the upper bound $\bbal(n,K_{1,k}) \leq a(n,k)$ and  $\Bbal(n,K_{1,k}) = \{H\}$, we must show that any $2$-edge coloring $E(K_{n,n})=R\sqcup B$, with at least $a(n,k)$ edges in each color and such that neither the red graph nor the blue graph is isomorphic to $H$, must contain a balanced copy of $K_{1,k}$. 

Let $V(K_{n,n})$ have the bipartition $V\sqcup W$. We define the following sets

$$V_r = \{ v\in V \mid \deg_R (v)\geq \frac{k}{2} \}, \hspace{30pt} V_b = \{ v\in V \mid \deg_B (v)\geq \frac{k}{2} \}.$$

$$W_r = \{ v\in W \mid \deg_R (v)\geq \frac{k}{2} \}, \hspace{30pt} W_b = \{ v\in W \mid \deg_B (v)\geq \frac{k}{2} \}.$$

Let $E(K_{n,n})=R\sqcup B$ be a $2$-edge coloring with at least $a(n,k)$ edges in each color and such that neither the red graph nor the blue graph is not isomorphic to $H$. If there is a vertex in $V_r \cap V_b$ or in $W_r \cap W_b$, then the graph induced by this vertex and its neighbors contains a balanced $K_{1,k}$. Therefore, we will assume that $V_r \cap V_b = \emptyset$ and $W_r \cap W_b = \emptyset$ and consequently that $V=V_r \sqcup V_b$ and $W=W_r \sqcup W_b$. This partition and the fact that any vertex $v\in V_R$ satisfies that $\deg_B (v) \leq \frac{k}{2}-1$ imply that $\deg_R (v)\geq n-\frac{k}{2}+1$ for all $v\in V_r$. Analogously, $\deg_B (v) \geq n-\frac{k}{2}+1$ for all $v\in V_b$, $\deg_R(w)\geq n-\frac{k}{2}+1$ for all $w\in W_r$, and $\deg_B (w)\geq n-\frac{k}{2}+1$ for all $w\in W_b$. We separate the proof in two cases.\\

\noindent\textbf{Case 1.} One of $\{V_r, V_b, W_r, W_b\}$ has at most $\frac{k}{2}-1$ elements. Without loss of generality, suppose $V_r$ to be such set, hence $|V_r|\leq \frac{k}{2}-1$. Now we bound the number of red edges.

\begin{align*}
a(k,n) \leq |R| & = e_R \left(V_r , W\right) + e_R (V_b, W)\\
& \leq |V_r|\cdot n + |V_b|(\frac{k}{2}-1)\\
& = |V_r|\cdot n + \left(n-|V_r|\right)\left(\frac{k}{2}-1\right)\\
& = n\cdot\left(\frac{k}{2}-1\right) + |V_r|\left(n-\frac{k}{2}+1\right)\\
& \leq n\cdot \left(\frac{k}{2}-1\right)+\left(\frac{k}{2}-1\right)\left(n-\frac{k}{2}+1\right)\\
& =a(k,n).
\end{align*}

Because the equality is reached, we can conclude the following:
\begin{itemize}
\item $\deg_R (v)=n$ for all $v\in V_r$, which implies that $\deg_B (v)=0$ for all $v\in V_r$;
\item $\deg_R(v)=\frac{k}{2}-1$ for all $v\in V_b$, which implies that $\deg_B (v)=n-\frac{k}{2}+1$ for all $v\in V_b$;
\item and $|V_r|=\frac{k}{2}-1$, which implies that $|V_b|=n-\frac{k}{2}+1$.
\end{itemize}

This also implies that every vertex in $W$ is connected to every vertex in $V_r$ by a red edge. Note that there cannot be any red edges connecting $V_b$ and $W_b$. Otherwise, if $vw$ is a red edge with $v\in V_b$ and $w\in W_b$ then the vertex $w$ would have more than $\frac{k}{2}-1$ red neighbors. Therefore, all edges connecting $V_b$ and $W_b$ must be blue. 

If $|W_r|\leq \frac{k}{2}-1$, then 

\begin{align*}
a(n,k) \leq |R| & = e_R (W_r , V) + e_R (W_b , V)\\
& \leq |W_r |\cdot n + \left(\frac{k}{2}-1\right)|W_b |\\
& = |W_r |\cdot n + \left(n- |W_r |\right)\left(\frac{k}{2}-1\right)\\
& = n \cdot \left(\frac{k}{2}-1\right) + |W_r|\left(n-\frac{k}{2}+1\right)\\
& \leq n \cdot \left(\frac{k}{2}-1\right) + \left(\frac{k}{2}-1\right)\left(n-\frac{k}{2}+1\right)\\
& = a(n,k).
\end{align*}

This implies that $|W_r |= \frac{k}{2}-1$, $\deg _R (w)=n$ for all $w\in W_r$ and $\deg_R (w)=\frac{k}{2}-1$ for all $w\in W_b$, which means that the red graph is isomorphic to $H$, a contradiction to the assumption we made at the beginning.

Therefore we can assume that $|W_r|=\frac{k}{2}-1+s$, for some $s\geq 1$. Because $\deg_B (v)=n-\frac{k}{2}+1$ for all $v\in V_b$ and all edges $vw$ with $v\in V_b$ and $w\in W_b$ are blue, then $e_B (v , W_r )=s$ for $v\in V_b$.

This implies that $e_B (W_r , V_b  ) =|V_b|\cdot s = (n-\frac{k}{2}+1)\cdot s$. Recall that $\deg_B (w)\leq \frac{k}{2}-1$ for all $w\in W_r$. We obtain the following inequality.

$$s\left(n-\frac{k}{2} +1\right) = e_B (W_r , V_b) \leq |W_r| \left(\frac{k}{2}-1\right)  = \left(\frac{k}{2}-1+s\right)\left(\frac{k}{2}-1\right)$$

Solving for $n$, we obtain

$$n \leq \left(\frac{k}{2}-1\right)^2 \frac{1}{s} + 2 \left(\frac{k}{2}-1\right) \leq \left(\frac{k}{2}-1\right)^2 + \left(k-2\right) =\frac{k^2}{4}-1. $$

This is a contradiction to the fact that $n\geq  \frac{k^2}{2}+k-2$.\\

\noindent\textbf{Case 2.} Suppose that $|V_r|, |V_b|,|W_r|,|W_b| \geq \frac{k}{2}$. Without loss of generality, we can assume that $|W_b|=\max \{|V_r |,|V_b |,|W_r |,|W_b |\}$. This implies that $|W_r|=\min \{|V_r |,|V_b |,|W_r |,|W_b |\}$. 

Let $|V_r |=|W_r|+t$ for some $t\geq 0$. We use the following two bounds on $|R|$.

$$|R| =e_R (V_r, W) + e_R (V_b ,W)\geq |V_r|\left(n-\frac{k}{2}+1\right) = \left(|W_r|+t\right)\left(n-\frac{k}{2}+1\right).$$
$$|R|= e_R (W_r, V) + e_r (W_b, V)\leq n \cdot |W_r|+ \left(\frac{k}{2}-1\right) |W_b|=n \cdot |W_r| + \left(\frac{k}{2}-1\right)\left(n-|W_r|\right).$$

Joining both bounds, we get 

$$n\left(t-\frac{k}{2}+1\right)\leq t \left(\frac{k}{2}-1\right).$$

Solving for $t$, we get the following.

$$t\leq \frac{n\left(\frac{k}{2}-1\right)}{n-\frac{k}{2}+1}=\frac{k}{2}-1+\frac{\left(\frac{k}{2}-1\right)^2}{n-\frac{k}{2}+1}<\frac{k}{2}$$

The last inequality holds if $n-\frac{k}{2}+1>(\frac{k}{2}-1)^2$, which is true due to the hypothesis that $n\geq \frac{k^2}{2}+k-2$.

This means that $t\leq \frac{k}{2}-1$. Note that $|V_b|=n-|V_r|=n-|W_r|-t$. We bound $e(V_b,W_r)$ by using the following.

$$e_B(W_r, V_b)\leq |W_r|\left( \frac{k}{2}-1\right);$$

$$e_R(V_b,W_r)\leq |V_b|\left(\frac{k}{2}-1\right).$$

The number of edges $e(V_b,W_r)$ is bounded as follows.

\begin{align*}
(n-|W_r|-t)|W_r|&=|V_b|\cdot |W_r|\\
&= e(V_b , W_r)\\
&\leq \left(|V_b|+ |W_r|\right)\left(\frac{k}{2}-1\right)\\
&= (n-t)(\frac{k}{2}-1).
\end{align*}

This is equivalent to:

$$n\left( |W_r|-\frac{k}{2}+1\right)\leq |W_r|^2 +t\cdot |W_r|-t\cdot (\frac{k}{2}-1).$$

Using the fact that $|W_r|-\frac{k}{2}+1>0$, we arrive at the next inequality for $n$. We leave a detailed discussion of the last inequality at the end of the proof.

\begin{equation}\label{eq-n}
n\leq \frac{|W_R|^2 +t( |W_r|-\frac{k}{2}+1)}{|W_r|-\frac{k}{2}+1}
=\frac{|W_r|^2}{|W_r|-\frac{k}{2}+1}+t\leq \frac{n}{2}+\frac{k^2}{4}-1.
\end{equation}

This means that $n\leq \frac{k^2}{2}-2$, which is a contradiction to $n\geq \frac{k^2}{2}+k-2$ for $k\geq2$.

We conclude the proof by justifying the last inequality in \Cref{eq-n}. Let $x=|W_r|$ and let $f$ be the following function depending on $x$, $f(x)=\frac{x^2}{x-\frac{k}{2}+1}$. The function $f(x)$ can be bounded as follows.

\begin{align*}
f(x)= \frac{x^2}{x-\frac{k}{2}+1}&=\frac{\left(x-\left(\frac{k}{2}-1\right)\right)^2+2x\left(\frac{k}{2}-1\right)-\left(\frac{k}{2}-1\right)^2}{x-\frac{k}{2}+1}\\
&= x-\frac{k}{2}+1 + \frac{\left(\frac{k}{2}-1\right)\left(x-\frac{k}{2}+1\right)+x\left(\frac{k}{2}-1\right)}{x-\frac{k}{2}+1}\\
&= x+\left(\frac{k}{2}-1\right)\frac{x}{x-\left(\frac{k}{2}-1\right)}\\
&\leq \frac{n}{2}+\left(\frac{k}{2}-1\right)\frac{\frac{k}{2}}{\frac{k}{2}-(\frac{k}{2}-1)}\\
&=\frac{n}{2}+\frac{k^2}{4}-\frac{k}{2}.
\end{align*}

The last inequality is due to the fact that $x\leq \frac{n}{2}$ and that $g(x)=\frac{x}{x-\left(\frac{k}{2}-1\right)}$ is a decreasing function defined on $\frac{k}{2}\leq x\leq\frac{n}{2}$. Therefore, the maximum value of $g(x)$ is reached when $x=\frac{k}{2}$. This bound on $f(x)$ and the fact that $t\leq \frac{k}{2}-1 $ provide the last inequality in \Cref{eq-n}.

$$n\leq \frac{|W_r|^2}{|W_r|-\frac{k}{2}+1}+t \leq \frac{n}{2}+\frac{k^2}{4}-\frac{k}{2}+ \frac{k}{2}-1=\frac{n}{2}+\frac{k^2}{4}-1.$$

\end{proof}

\section*{Acknowledgements}
This project was supported by PAPIIT IG100822.





\bibliographystyle{abbrv}
\bibliography{biblio}

\end{document}